\documentclass[10pt]{amsart}
\usepackage{amsfonts}
\usepackage{ifthen}
\usepackage{amsthm}
\usepackage{amsmath}
\usepackage{graphicx}
\usepackage{amscd,amssymb,amsthm}

\newcounter{minutes}
\divide\time by 60
\newcounter{hours}
\multiply\time by 60 \addtocounter{minutes}{-\time}

\newcommand{\real}{\operatorname{Re}}
\newcommand{\logo}{\operatorname{Log}}

\newtheorem{Theorem}{Theorem}
\newtheorem{lemma}{Lemma}


\title[\'A. Baricz, R. Sz\'asz/The radius of convexity of Bessel functions]{The radius of convexity of normalized Bessel functions of the first kind}

\author[]{\'Arp\'ad Baricz}
\address{Department of Economics, Babe\c{s}-Bolyai University, 400591 Cluj-Napoca, Romania} \email{bariczocsi@yahoo.com}

\author[]{R\'obert Sz\'asz}
\address{Department of Mathematics and Informatics, Sapientia Hungarian University of Transylvania, 540485 T\^argu-Mure\c{s}, Romania}
\email{rszasz@ms.sapientia.ro}

\thanks{A part of this paper was presented by \'A. Baricz at the international conference Computational Methods and Function Theory, which was held in Shantou, China, between June 10-14, 2013. The research of \'A. Baricz was supported by a research grant of the Romanian National Authority for Scientific Research, CNCS-UEFISCDI, project number PN-II-RU-TE-2012-3-0190/2014.}

\begin{document}

\def\thefootnote{}
\footnotetext{ \texttt{File:~\jobname .tex,
          printed: \number\year-0\number\month-0\number\day,
          \thehours.\ifnum\theminutes<10{0}\fi\theminutes}
} \makeatletter\def\thefootnote{\@arabic\c@footnote}\makeatother

\maketitle

\begin{abstract}
In this paper we determine the radius of convexity for three kind of
normalized Bessel functions of the first kind. In the mentioned cases the normalized Bessel
functions are starlike-univalent and convex-univalent, respectively, on
the determined disks. The key tools in the proofs of the main results are some new Mittag-Leffler expansions for quotients of Bessel functions of the first kind, special properties of the zeros of Bessel functions of the first kind and their derivative, and the fact that the smallest positive zeros of some Dini functions are less than the first positive zero of the Bessel function of the first kind. Moreover, we find the optimal parameters for which these normalized Bessel functions are convex in the open unit disk. In addition, we disprove a conjecture of Baricz and Ponnusamy concerning the convexity of the Bessel function of the first kind.
\end{abstract}

\section{\bf Introduction and Main Results}
Let $\mathbb{D}(z_0,r)=\left\{z\in\mathbb{C}:|z-z_0|<r\right\}$ denote the open disk centered
in $z_0$ and of radius $r>0,$ and let $\mathcal{S}$ be the class
of analytic and univalent functions defined in the open unit disk
$\mathbb{D}=\left\{z\in\mathbb{C}:|z|<1\right\}$ and having the property that $f(0)=f'(0)-1=0$.
Recall that a function $f\in\mathcal{S}$ belongs to the class $\mathcal{K}$ of convex functions if maps the unit disk conformally onto $f(\mathbb{D}),$ which is a convex domain in $\mathbb{C},$ that is, the domain $f(\mathbb{D})\subset\mathbb{C}$ contains the entire line segment joining any pair of its points. It is well-known that the class of convex functions can be characterized as
 $$\mathcal{K} = \left\{f\in \mathcal{S} \left| \real \left(1+\frac{zf''(z)}{f'(z)}\right) > 0, \ z\in \mathbb{D}\right. \right\}.$$
 Moreover, for $\alpha\in[0,1)$ we consider also the class of convex functions of order
 $\alpha$ defined by
 $$\mathcal{K}(\alpha) = \left\{f\in \mathcal{S} \left| \real \left(1+\frac{zf''(z)}{f'(z)}\right) > \alpha, \ z\in \mathbb{D} \right. \right\}.$$
Now let us consider the radius of convexity, and the radius of convexity of order $\alpha$ of the function $f$
$$r^c(f)=\sup\left\{r\in(0,\infty)\left|  \real \left(1+\frac{zf''(z)}{f'(z)}\right) >0,\ \ z\in\mathbb{D}(0,r)\right.\right\}$$
and
$$r^c_\alpha(f)=\sup\left\{r\in(0,\infty)\left|  \real \left(1+\frac{zf''(z)}{f'(z)}\right) >\alpha,\ \ z\in\mathbb{D}(0,r)\right.\right\}.$$
We note that $r^c(f)$ is in fact the largest radius for which the image domain $f(\mathbb{D}(0,r^c(f)))$ is a convex domain in $\mathbb{C}.$

The Bessel function of the first kind of order $\nu$ is defined by \cite[p. 217]{nist}
$$J_{\nu}(z)=\sum_{n\geq0}\frac{(-1)^{n}}{n!\Gamma(n+\nu+1)}\left(\frac{z}{2}\right)^{2n+\nu}.$$
In this paper we focus on the following  normalized forms
$$f_{\nu}(z)=\left(2^{\nu}\Gamma(\nu+1)J_{\nu}(z)\right)^{\frac{1}{\nu}}=z-\frac{1}{4\nu(\nu+1)}z^3+\dots,\ \nu\neq0, $$
$$g_{\nu}(z)=2^{\nu}\Gamma(\nu+1)z^{1-{\nu}}J_{\nu}(z)=z-\frac{1}{4(\nu+1)}z^3+\frac{1}{32(\nu+1)(\nu+2)}z^5-\dots,$$
$$h_{\nu}(z)=2^{\nu}\Gamma(\nu+1)z^{1-\frac{\nu}{2}}J_{\nu}(\sqrt{z})=z-\frac{1}{4(\nu+1)}z^2+\dots,$$
where $\nu>-1.$ We note that $$f_{\nu}(z)=\exp\left(\frac{1}{\nu}\logo\left(2^{\nu}\Gamma(\nu+1)J_{\nu}(z)\right)\right),$$
where $\logo$ represents the principal branch of the logarithm, and in this paper every
multi-valued function is taken with the principal branch. We also mention that the univalence, starlikeness and convexity of other functions involving the Bessel function of the first kind were studied extensively in several papers. We refer to \cite{mathematica,publ,lecture,bsk,samy,brown,todd,szasz,szasz2} and to the references therein.

In this paper we make a further contribution to the subject by showing the following new sharp results. The next section contains some preliminary results, and the proofs of Theorems \ref{th1}, \ref{th2} and \ref{th3} can be found in section 3.

\begin{Theorem}\label{th1}
If $\nu>0$ and $\alpha\in[0,1),$  then the radius of
convexity of order $\alpha$ of the function $f_\nu$ is the smallest positive
root of the equation
$$
1+\frac{rJ_\nu''(r)}{J_\nu'(r)}+\left(\frac{1}{\nu}-1\right)\frac{rJ_\nu'(r)}{J_\nu(r)}=\alpha.
$$
Moreover, $r^{c}_{\alpha}(f_{\nu})<j_{\nu,1}'<j_{\nu,1},$ where $j_{\nu,1}$ and $j_{\nu,1}'$ denote the first positive zeros of $J_{\nu}$ and $J_{\nu}',$ respectively.
\end{Theorem}

\begin{Theorem}\label{th2}
If $\nu>-1$ and  $\alpha\in[0,1),$ then the radius of
convexity of order $\alpha$ of the function $g_{\nu}$ is the smallest positive root of
the equation
$$
 1+r\frac{rJ_{\nu+2}(r)-3J_{\nu+1}(r)}{J_{\nu}(r)-rJ_{\nu+1}(r)}=\alpha.
$$
Moreover, we have $r^c_{\alpha}(g_{\nu})<\alpha_{\nu,1}<j_{\nu,1},$ where $\alpha_{\nu,1}$ is the first positive zero of the Dini function $z\mapsto (1-\nu)J_\nu(z)+zJ'_\nu(z).$
 \end{Theorem}

\begin{Theorem}\label{th3}
If $\nu>-1$ and $\alpha\in[0,1),$  then the radius of
convexity of order $\alpha$ of the function $h_{\nu}$ is the smallest  positive root of the equation
$$
 1+\frac{r^{\frac{1}{2}}}{2}\cdot\frac{r^{\frac{1}{2}}J_{\nu+2}(r^{\frac{1}{2}})-
 4J_{\nu+1}(r^{\frac{1}{2}})}{2J_{\nu}(r^{\frac{1}{2}})-r^{\frac{1}{2}}J_{\nu+1}(r^{\frac{1}{2}})}=\alpha.
$$
Moreover, we have $r^c_{\alpha}(h_{\nu})<\beta_{\nu,1}<j_{\nu,1},$ where $\beta_{\nu,1}$ is the first positive zero of the Dini function $z\mapsto (2-\nu)J_\nu(z)+zJ'_\nu(z).$
 \end{Theorem}

The real number
$$r^{*}(f)=\sup\left\{r\in(0,\infty)\left| \real\left(\frac{zf'(z)}{f(z)}\right)\right. > 0, \ \
z\in \mathbb{D}(0,r)\right\}$$
is called the radius of starlikeness of the function $f$ and it is the largest radius such that
$f\left(\mathbb{D}(0,r^{*}(f))\right)$ is a starlike domain with
respect to $0.$ It is important to mention here that the problem on the radius of starlikeness of the functions $f_{\nu}$ and $g_{\nu}$ was first studied by Brown \cite{brown} who found these radii in the case $\nu>0.$ Recently, the authors and Kup\'an \cite{bsk} used a different approach to find the radii of starlikeness of order $\beta\in[0,1)$ for the functions $f_{\nu}$ and $g_{\nu}.$ We note that for $\nu>0$ the radius of starlikeness of $f_{\nu}$ is $j_{\nu,1}',$ while for $\nu>-1$ the radius of starlikeness of $g_{\nu}$ is $\alpha_{\nu,1},$ which was mentioned above in Theorem \ref{th2}. Note that Brown \cite{brown} used the methods of Nehari \cite{nehari} and Robertson \cite{robertson}, and the key tool in Brown's proofs was the fact that the Bessel function of the first kind is a particular solution of the Bessel differential equation. For related results the interested reader is referred to \cite{brown2,brown3,merkes,robertson,wilf} and to the references therein, and for more details we refer to \cite{bsk}. Our approach is completely different than of Brown \cite{brown,brown2,brown3}, Nehari \cite{nehari} and Robertson \cite{robertson}. The key tools in the proofs of our main results are some new Mittag-Leffler type expansions for Bessel functions of the first kind, properties of the zeros of Bessel functions, and the fact that the smallest positive zeros of some Dini function are less than the first positive zero of the Bessel function of the first kind.

Now, we are going to present some other sharp results on the functions $f_{\nu},$ $g_{\nu}$ and $h_{\nu}.$ The proofs of these results can be found in section 3.

\begin{Theorem}\label{th4} The function $f_{\nu}$ is convex of order $\alpha\in(0,1)$ in $\mathbb{D}$ if and only if
$\nu\geq\nu_{\alpha}(f_{\nu}),$ where $\nu_{\alpha}(f_{\nu})$ is the unique  root of
the equation $$\nu(\nu^2-1)J_{\nu}^2(1)+(1-\nu)(J_{\nu}'(1))^2=\alpha\nu J_{\nu}(1)J_{\nu}'(1),$$ situated in $(\nu^{\ast},\infty),$ where $\nu^{\ast}\simeq0.3901\dots$ is the root of the equation $J_{\nu}'(1)=0.$ Moreover, $f_{\nu}$ is convex in $\mathbb{D}$ if and only if $\nu\geq1.$
\end{Theorem}

\begin{Theorem}\label{th5} The function $g_{\nu}$ is convex of order $\alpha\in[0,1)$ in $\mathbb{D}$ if and only if
$\nu\geq\nu_{\alpha}(g_{\nu}),$ where $\nu_{\alpha}(g_{\nu})$ is the unique  root of
the equation $$(2\nu+\alpha-2)J_{\nu+1}(1)=\alpha J_{\nu}(1),$$ situated in $[0,\infty).$
In particular, $g_{\nu}$ is convex in $\mathbb{D}$ if and only if
$\nu\geq1.$ \end{Theorem}

\begin{Theorem}\label{th6} The function $h_{\nu}$ is convex of order $\alpha\in[0,1)$ in $\mathbb{D}$ if and only if
$\nu\geq\nu_{\alpha}(h_{\nu}),$ where $\nu_{\alpha}(h_{\nu})$ is the unique  root of
the equation $$(2\nu+2\alpha-4)J_{\nu+1}(1)=(4\alpha-3)J_{\nu}(1),$$
situated in $[0,\infty).$
In particular, $h_{\nu}$ is convex if and only if
$\nu\geq\nu_0(h_{\nu}),$ where $\nu_0(h_{\nu})\simeq-0.1438\dots$ is the unique root of the equation
$$(2\nu-4)J_{\nu+1}(1)+3J_{\nu}(1)=0.$$ Moreover, in particular, the function $h_{\nu}$ is convex of order $\frac{3}{4}$ if and only if $\nu\geq\frac{5}{4}.$
\end{Theorem}

We note that the convex functions does not need to be normalized. In other words, the analytic and univalent function $f:\mathbb{D}\to\mathbb{C}$ satisfying $f'(0)\neq0$ is said to be convex of order $\alpha\in[0,1)$ if and only if
$$\real \left(1+z\frac{f''(z)}{f'(z)}\right)>\alpha$$ for all $z\in\mathbb{D}.$ In 1995 Selinger \cite{selinger} by using the method of differential subordinations proved that the function $\varphi_{\nu}:\mathbb{D}\to\mathbb{C},$ defined by $$\varphi_{\nu}(z)=\frac{h_{\nu}(z)}{z}=2^{\nu}\Gamma(\nu+1)z^{-\frac{\nu}{2}}J_{\nu}(\sqrt{z})=1-\frac{1}{4(\nu+1)}z+\dots,$$
is convex if $\nu\geq-\frac{1}{4}.$ In 2009 Sz\'asz and Kup\'an \cite{szasz2}, by using a completely different approach, improved this result, and proved that $\varphi_{\nu}$ is convex in $\mathbb{D}$ if $\nu\geq \nu_1\simeq -1.4069\dots,$ where $\nu_1$ is the root of the equation $4\nu^2+17\nu+16=0.$ Recently, Baricz and Ponnusamy \cite{samy} presented four improvements of the above result, and their best result was the following \cite[Theorem 3]{samy}: the function $\varphi_{\nu}$ is convex in $\mathbb{D}$ if $\nu\geq \nu_2\simeq-1.4373\dots,$ where $\nu_2$ is the unique root of the equation $2^{\nu}\Gamma(\nu+1)(I_{\nu+2}(1)+2I_{\nu+1}(1))=2.$ Moreover, Baricz and Ponnusamy \cite{samy} conjectured that $\varphi_{\nu}$ is convex in $\mathbb{D}$ if and only if $\nu\geq-1.875.$ Now, we are able to disprove this conjecture and to find the radius of convexity of the function $\varphi_{\nu}.$

\begin{Theorem}\label{th7}
If $\nu>-2$ and $\alpha\in[0,1),$  then the radius of
convexity of order $\alpha$ of the function $\varphi_{\nu}$ is the smallest  positive root of the equation
$$
 \frac{r^{\frac{1}{2}}J_{\nu}(r^{\frac{1}{2}})}{2J_{\nu+1}(r^{\frac{1}{2}})}-\nu=\alpha.
$$
Moreover, we have $r^c_{\alpha}(\varphi_{\nu})<j_{\nu+1,1}.$
\end{Theorem}

\begin{Theorem}\label{th8}
The function $\varphi_{\nu}$ is convex of order $\alpha\in[0,1)$ in $\mathbb{D}$ if and only if
$\nu\geq\nu_{\alpha}(\varphi_{\nu}),$ where $\nu_{\alpha}(\varphi_{\nu})$ is the unique  root of
the equation $$(2\nu+2\alpha)J_{\nu+1}(1)= J_{\nu}(1),$$ situated in $\left(\nu^{\star},\infty\right),$ where $\nu^{\star}\simeq-1.7744\dots$ is the root of the equation $J_{\nu+1}(1)=0.$ In particular, $\varphi_{\nu}$ is convex in $\mathbb{D}$ if and only if $\nu\geq \nu_0(\varphi_{\nu}),$ where $\nu_0(\varphi_{\nu})\simeq-1.5623\dots$ is the unique root of the equation $J_{\nu}(1)=2\nu J_{\nu+1}(1),$ situated in $\left(\nu^{\star},\infty\right).$
\end{Theorem}

\section{\bf Preliminary Results}
\setcounter{equation}{0}

This section is devoted to present some preliminary results, which will be used to prove the main theorems. Some of these preliminary results are well-known, however, Lemma \ref{lem4} and \ref{lem5} are quite new, and may be of independent interest.

\begin{lemma}\label{lem0}
If $a>b>0,$ $z\in\mathbb{C}$ and $\lambda\in[0,1],$ then for all $|z|<b$ we have
\begin{equation}\label{llll}\lambda\real\left(\frac{z}{a-z}\right)-\real\left(\frac{z}{b-z}\right)\geq\lambda\frac{|z|}{a-|z|}-\frac{|z|}{b-|z|}.\end{equation}
\end{lemma}

\begin{proof}[\bf Proof]
Let us consider the function $u:[b,\infty)\rightarrow{\mathbb{R}},$ defined
by $$u(t)=\real\left(\frac{z}{t-z}\right)-\frac{|z|}{t-|z|}.$$ Simple
computations lead to
$$u(t)=\frac{tx-m^2}{t^2-2xt+m^2}-\frac{m}{t-m},$$
$$u'(t)=\frac{2tm^2-t^2x-xm^2}{(t^2-2xt+m^2)^2}+\frac{m}{(t-m)^2},$$
where $z=x+\mathrm{i}y$ and $|z|=m.$ Since for $t\geq b$ we have
$$u'(t)\geq\frac{2tm^2-t^2x-xm^2}{(t^2-2xt+m^2)^2}+\frac{m}{t^2-2xt+m^2}=\frac{(m-x)(t+m)^2}{(t^2-2xt+m^2)^2}>0,$$
it follows that $u$ is an increasing function, and consequently we get $u(a)\geq{u}(b),$ which is equivalent to
\begin{equation}\label{klkl}\real\left(\frac{z}{a-z}\right)-\frac{|z|}{a-|z|}\geq\real\left(\frac{z}{b-z}\right)-\frac{|z|}{b-|z|}.\end{equation}
On the other hand, it is known (see \cite{szasz}) that if $z\in{\mathbb{C}}$ and $\mu\in{\mathbb{R}}$ such that
$\mu>|z|$, then
\begin{equation}\label{kk}\frac{|z|}{\mu-|z|}\geq \real\left(\frac{z}{\mu-z}\right).\end{equation}
This in turn implies that
$$\lambda\left(\real\left(\frac{z}{a-z}\right)-\frac{|z|}{a-|z|}\right)\geq\real\left(\frac{z}{a-z}\right)-\frac{|z|}{a-|z|}.$$
and combining this with \eqref{klkl} we get \eqref{llll}.
\end{proof}

\begin{lemma}\label{lem1}\cite[p. 482]{watson}
If $\nu>-1$  and $a,b\in{\mathbb{R}},$ then $z\mapsto aJ_{\nu}(z)+bzJ'_{\nu}(z)$ has all its zeros real,
except the case when $a/b+\nu<0$. In this case it has
two purely imaginary zeros beside the real roots. Moreover, if $\nu>-1$ and $a,b\in\mathbb{R}$ such
that $a^2+b^2\neq0,$ then no function of the type
$z\mapsto aJ_\nu(z)+bzJ'_\nu(z)$ can have a repeated zero other than
$z=0.$
\end{lemma}

\begin{lemma}\label{lem3}\cite[p. 198]{watson}
If $\nu>-1,$ then for the Bessel functions of the third kind $H_\nu^{(1)}$ and $H_\nu^{(2)}$
the following asymptotic expansions are valid
$$H_\nu^{(1)}(w)=\left(\frac{2}{\pi{w}}\right)^{\frac{1}{2}}e^{\mathrm{i}(w-\frac{1}{2}\nu\pi-\frac{1}{4}\pi)}(1+\eta_{1,\nu}(w)),$$
$$H_\nu^{(2)}(w)=\left(\frac{2}{\pi{w}}\right)^{\frac{1}{2}}e^{-\mathrm{i}(w-\frac{1}{2}\nu\pi-\frac{1}{4}\pi)}(1+\eta_{2,\nu}(w)),$$
where $\eta_{1,\nu}(w)$ and $\eta_{2,\nu}(w)$ are $\mathcal{O}(1/w)$ when $|w|$ is large.
\end{lemma}

\begin{lemma}\label{lem4}
Let $z\in\mathbb{C}$, and $\alpha_{\nu,n}$ be the $n$th positive root of the
equation $J_{\nu}(z)-zJ_{\nu+1}(z)=0.$ If $\nu>-1,$ then the
following development holds
\begin{equation}\label{sxv}\frac{g_\nu''(z)}{g_\nu'(z)}=
\frac{zJ_{\nu+2}(z)-3J_{\nu+1}(z)}{J_{\nu}(z)-zJ_{\nu+1}(z)}=
-\sum_{n\geq1}\frac{2z}{\alpha_{\nu,n}^2-z^2}.\end{equation}
\end{lemma}

\begin{proof}[\bf Proof]
Let us consider the integral
$$\frac{1}{2\pi\mathrm{i}}\int_{\mathbb{U}}\frac{z}{w(w-z)}\frac{wJ_{\nu+2}(w)-3J_{\nu+1}(w)}{J_{\nu}(w)-wJ_{\nu+1}(w)}dw$$
where $\mathbb{U}$ is the rectangle, whose vertices  are $\pm{a}\pm{b\mathrm{i}},$ $a>0,$ $b>0$ and $z$ is a point inside the rectangle $\mathbb{U}$ other
than a zero of $z\mapsto J_{\nu}(z)-zJ_{\nu+1}(z).$ Suppose that inside of $\mathbb{U}$ there are $m$ positive and $m$ negative roots of $z\mapsto J_{\nu}(z)-zJ_{\nu+1}(z).$ Since \cite[p. 222]{nist} $(1-\nu)J_\nu(z)+zJ'_\nu(z)=J_\nu(z)-zJ_{\nu+1}(z),$ according to Lemma \ref{lem1} the zeros of $z\mapsto J_{\nu}(z)-zJ_{\nu+1}(z)$ are simple and real. The point $w=0$ is a removable singularity. The only poles of the above integrand inside the rectangle are $z,$ $\pm \alpha_{\nu,1},$ $\pm \alpha_{\nu,2},$ $\dots,$ $\pm \alpha_{\nu,m}.$ The residue at $z$ is
$$\varphi_{\nu}(z)=\frac{zJ_{\nu+2}(z)-3J_{\nu+1}(z)}{J_{\nu}(z)-zJ_{\nu+1}(z)},$$
while the residues at $\pm \alpha_{\nu,n}$ are $$\frac{z}{\alpha_{\nu,n}(\alpha_{\nu,n}\mp z)}.$$
Here we used the recurrence relation \cite[p. 222]{nist}
$$zJ_{\nu}'(z)=-zJ_{\nu+1}(z)+\nu J_{\nu}(z)$$
for $\nu$ and $\nu+1.$ The residue theorem \cite[p. 19]{nist} implies that
\begin{align*}\frac{1}{2\pi\mathrm{i}}&\int_{\mathbb{U}}\frac{z}{w(w-z)}\frac{wJ_{\nu+2}(w)-3J_{\nu+1}(w)}{J_{\nu}(w)-wJ_{\nu+1}(w)}dw\\&=
\frac{zJ_{\nu+2}(z)-3J_{\nu+1}(z)}{J_{\nu}(z)-zJ_{\nu+1}(z)}+\sum_{n=1}^m\frac{z}{\alpha_{\nu,n}(\alpha_{\nu,n}-z)}+
\sum_{n=1}^m\frac{z}{\alpha_{\nu,n}(\alpha_{\nu,n}+z)}\\&=\frac{zJ_{\nu+2}(z)-3J_{\nu+1}(z)}{J_{\nu}(z)-zJ_{\nu+1}(z)}+
\sum_{n=1}^m\frac{2z}{\alpha_{\nu,n}^2-z^2}.\end{align*} In what follows we show that $a$ and $b$ can be replaced by suitable
sequences which increase without limit and in the same time the
expression $\varphi_{\nu}(w)$ remains bounded if $w\in\mathbf{U}.$
Since the function $w\mapsto \varphi_{\nu}(w)$ is
an odd function of $w,$  it is sufficient to consider the case
$\real{w}>0.$ Now, we introduce the notations
$$H_{\nu+k}^{(j)}(w)=\left(\frac{2}{\pi{w}}\right)^{\frac{1}{2}}e^{(-1)^{j-1}\mathrm{i}(w-\frac{1}{2}(\nu+k)\pi-\frac{1}{4}\pi)}(1+\eta_{k,j,\nu}(w)),$$
where, as in Lemma \ref{lem3}, for $k\in\{0,1,2\}$ and $j\in\{1,2\}$ the expression $\eta_{k,j,\nu}(w)$ is $\mathcal{O}(1/w)$ when $|w|$ is large. The relation $2J_\nu(w)=H_\nu^{(1)}(w)+H_\nu^{(2)}(w)$ and Lemma \ref{lem3}
 lead to $\varphi_{\nu}(w)={p_{\nu}(w)}/{q_{\nu}(w)},$ where
\begin{eqnarray}p_{\nu}(w)=(-\mathrm{i})e^{2\mathrm{i}(w-\frac{1}{2}(\nu+1)\pi-\frac{1}{4}\pi)}(1+\eta_{2,1,\nu}(w))+\mathrm{i}(1+\eta_{2,2,\nu}(w))\nonumber\\
-\frac{3}{w}\left[e^{2\mathrm{i}(w-\frac{1}{2}(\nu+1)\pi-\frac{1}{4}\pi)}(1+\eta_{1,1,\nu}(w))
+(1+\eta_{1,2,\nu}(w))\right] \nonumber \end{eqnarray} and
\begin{eqnarray}q_{\nu}(w)=\frac{1}{w}\left[\mathrm{i}e^{2\mathrm{i}(w-\frac{1}{2}(\nu+1)\pi-\frac{1}{4}\pi)}(1+\eta_{0,1,\nu}(w)) -\mathrm{i}(1+\eta_{0,2,\nu}(w))\right]\nonumber\\
-e^{2\mathrm{i}(w-\frac{1}{2}(\nu+1)\pi-\frac{1}{4}\pi)}(1+\eta_{1,1,\nu}(w))
+(1+\eta_{1,2,\nu}(w)).\nonumber \end{eqnarray}
Since $\varphi_{\nu}(x+\mathrm{i}b)$ tends to $\mathrm{i}$ as $b\to\infty$ and the
convergence is uniform with respect to $x\in\mathbb{R},$ it follows that $\varphi_{\nu}(w)$ is
bounded on $\left\{x+\mathrm{i}b|x\in[0,a]\right\}$ if $b$ is large enough. An
analogous argument shows that $\varphi_{\nu}(w)$ is
bounded on the segment $\left\{x-\mathrm{i}b|x\in[0,a]\right\}.$ It remains to prove
that $\varphi_{\nu}(w)$ is bounded in the case when $w\in\left\{a+\mathrm{i}y|y\in[-b,b]\right\}.$ In this case
we put $a=a_m=2m\pi+\frac{1}{2}(\nu+1)+\frac{3}{4}\pi$ and we get
$$\lim_{m\rightarrow\infty}\frac{p_{\nu}(2m\pi+\frac{1}{2}(\nu+1)+\frac{3}{4}\pi+\mathrm{i}y)}
{q_{\nu}(2m\pi+\frac{1}{2}(\nu+1)+\frac{3}{4}\pi+\mathrm{i}y)}=\frac{\mathrm{i}(1+e^{-2y})}{1+e^{-2y}}=\mathrm{i},$$
and thus it follows that
$\varphi_{\nu}(w)$ is bounded on the segment $\left\{a+\mathrm{i}y|y\in[-b,b]\right\},$ if $a_m$ is
large enough. Consequently $\varphi_{\nu}(w)$ is bounded on the perimeter of the rectangle $\mathbb{U}$ if $b$ and $a_m$
tend to infinity. Thus it follows that
$$\lim\sb{\substack{{b\rightarrow\infty}\\
{m\rightarrow\infty}}}\frac{1}{2\pi\mathrm{i}}\int_{\mathbb{U}}\frac{z}{w(w-z)}\frac{wJ_{\nu+2}(w)-3J_{\nu+1}(w)}{J_{\nu}(w)-wJ_{\nu+1}(w)}dw
=0,$$ which implies that (\ref{sxv}) is indeed valid.
\end{proof}

The proof of the next lemma is quite similar to that of the proof of Lemma \ref{lem4}. However, for the sake of completeness we have included also in details the proof of the following lemma. As far as we know the results presented in Lemma \ref{lem4} and \ref{lem5} are new. These results may be of independent interest and we believe that can be used to obtain some new inequalities for Bessel functions of the first kind of real or complex variable and real order.

\begin{lemma}\label{lem5}
Let $z\in\mathbb{C},$ and let $\beta_{\nu,n}$ be the $n$th positive root of $(2-\nu)J_\nu(z)+zJ_\nu'(z)=0.$ If $\nu>-1,$  then
the following development holds
$$\frac{\nu(\nu-2)J_\nu(z^{\frac{1}{2}})+(3-2\nu)z^{\frac{1}{2}}J_\nu'(z^{\frac{1}{2}})+zJ_\nu''(z^{\frac{1}{2}})}
{2(2-\nu)J_\nu(z^{\frac{1}{2}})+2z^{\frac{1}{2}}J_\nu'(z^{\frac{1}{2}})}=-\sum_{n\geq1}\frac{z}{\beta_{\nu,n}^2-z}.$$
\end{lemma}

\begin{proof}[\bf Proof]
We prove first the next development formula
$$\frac{\nu(\nu-2)J_\nu(z)+(3-2\nu)zJ_\nu'(z)+z^2J_\nu''(z)}{2(2-\nu)J_\nu(z)+2zJ_\nu'(z)}=-\sum_{n\geq1}\frac{z^2}{\beta_{\nu,n}^2-z^2}.$$
The recurrence formula \cite[p. 222]{nist} $zJ'_\nu(z)=\nu{J_\nu(z)}-zJ_{\nu+1}(z)$
implies
$$(2-\nu)J_\nu(z)+zJ_\nu'(z)=2J_{\nu}(z)-zJ_{\nu+1}(z)$$
and
$$\frac{\nu(\nu-2)J_\nu(z)+(3-2\nu)zJ_\nu'(z)+z^2J_\nu''(z)}{2(2-\nu)J_\nu(z)+2zJ_\nu'(z)}=\frac{z}{2}\cdot\frac{zJ_{\nu+2}(z)-4J_{\nu+1}(z)}{2J_{\nu}(z)-zJ_{\nu+1}(z)}.$$
We consider the integral
$$\frac{1}{2\pi{i}}\int_{\mathbb{O}}\frac{z}{w(w-z)}\frac{wJ_{\nu+2}(w)-4J_{\nu+1}(w)}{2J_{\nu}(w)-wJ_{\nu+1}(w)}dw$$
where $\mathbb{O}$ is the rectangle, whose vertices  are  $\pm{A}\pm{Bi} \
(A>0, \ B>0)$ and  $z$ is a point inside the rectangle $\mathbb{O}$ other
than a zero of $2J_{\nu}(z)-zJ_{\nu+1}(z).$ Moreover, we assume that inside of $\mathbb{O}$ there
 are $m$ positive and $m$ negative roots of $2J_{\nu}(z)-zJ_{\nu+1}(z).$ Of course we can make this assumption since \cite[p. 222]{nist} $(2-\nu)J_\nu(z)+zJ'_\nu(z)=2J_\nu(z)-zJ_{\nu+1}(z)$ and according to Lemma \ref{lem1} the zeros of $2J_{\nu}(z)-zJ_{\nu+1}(z)$ are simple and real.
Observe that the point $w=0$ is a removable singularity, and the residue theorem \cite[p. 19]{nist} implies that
\begin{align}\label{a4}\frac{1}{2\pi{i}}&\int_{\mathbb{O}}\frac{z}{w(w-z)}\frac{wJ_{\nu+2}(w)-4J_{\nu+1}(w)}{2J_{\nu}(w)-wJ_{\nu+1}(w)}dw\nonumber\\&
=\frac{zJ_{\nu+2}(z)-4J_{\nu+1}(z)}{2J_{\nu}(z)-zJ_{\nu+1}(z)}+\sum_{n=1}^m\frac{z}{\beta_{\nu,n}(\beta_{\nu,n}-z)}+
\sum_{n=1}^m\frac{z}{\beta_{\nu,n}(\beta_{\nu,n}+z)}\nonumber\\&=\frac{zJ_{\nu+2}(z)-4J_{\nu+1}(z)}{2J_{\nu}(z)-zJ_{\nu+1}(z)}+
\sum_{n=1}^m\frac{2z}{\beta_{\nu,n}^2-z^2}.\end{align}
In what follows we show that $A$ and $B$ can be replaced by suitable
sequences which increase without limit and in the same time the
expression
 $$\frac{wJ_{\nu+2}(w)-4J_{\nu+1}(w)}{2J_{\nu}(w)-wJ_{\nu+1}(w)}$$
 remains bounded if $w\in{\mathbb{O}}.$
Since the above function is an odd function of $w,$ it is sufficient to consider the case when $\real{w}>0.$ As in Lemma \ref{lem4} we use the notations
$$H_{\nu+k}^{(j)}(w)=\Big(\frac{2}{\pi{w}}\Big)^{\frac{1}{2}}e^{(-1)^{j-1}i(w-\frac{1}{2}(\nu+k)\pi-\frac{1}{4}\pi)}(1+\eta_{k,j,\nu}(w)),$$
where $k\in\{0,1,2\}$ and $j\in\{1,2\}.$ The relation $$2J_\nu(w)=H_\nu^{(1)}(w)+H_\nu^{(2)}(w)$$ and Lemma \ref{lem3}
lead to
$$\frac{wJ_{\nu+2}(w)-4J_{\nu+1}(w)}{2J_{\nu}(w)-wJ_{\nu+1}(w)}=\frac{S(w)}{N(w)},$$
 where
\begin{eqnarray}S(w)=(-i)e^{2i(w-\frac{1}{2}(\nu+1)\pi-\frac{1}{4}\pi)}(1+\eta_{2,1,\nu}(w)) +i(1+\eta_{2,2,\nu}(w))\nonumber\\
-\frac{4}{w}\big[e^{2i(w-\frac{1}{2}(\nu+1)\pi-\frac{1}{4}\pi)}(1+\eta_{1,1,\nu}(w))
+(1+\eta_{1,2,\nu}(w))\big] \nonumber \end{eqnarray} and
\begin{eqnarray}N(w)=\frac{2}{w}\big[ie^{2i(w-\frac{1}{2}(\nu+1)\pi-\frac{1}{4}\pi)}(1+\eta_{0,1,\nu}(w)) -i(1+\eta_{0,2,\nu}(w))\big]\nonumber\\
-e^{2i(w-\frac{1}{2}(\nu+1)\pi-\frac{1}{4}\pi)}(1+\eta_{1,1,\nu}(w))
+(1+\eta_{1,2,\nu}(w)).\nonumber \end{eqnarray} Since
$$\lim_{B\rightarrow\infty}\frac{S(x+iB)}{N(x+iB)}=i$$ and the
convergence is uniform with respect to  $x\in\mathbb{R},$ it
follows that $S(w)/N(w)$ is bounded on $\{x+iB|x\in[0,A]\}$ if $B$ is  large  enough. An
analogous argument shows that $S(w)/N(w)$ is bounded on the segment  $\{x-iB|x\in[0,A]\}.$ It remains to prove
the boundness in  case  $w\in\{A+iy|y\in[-B,B]\}.$ In this case
we put $A=A_m=2m\pi+\frac{1}{2}(\nu+1)+\frac{3}{4}\pi$  and we get
\begin{equation}\label{wm}\lim_{m\rightarrow\infty}\frac{S(2m\pi+\frac{1}{2}(\nu+1)+\frac{1}{4}\pi+iy)}{N(2m\pi+\frac{1}{2}(\nu+1)+\frac{1}{4}\pi+iy)}=i\end{equation}
The convergence in (\ref{wm}) is uniform on $y\in[-B,B],$  it follows that $S(w)/N(w)$ is bounded on the segment  $\{A_m+iy|y\in[-B,B]\},$  if $A_m$ is
large enough. Consequently $S(w)/N(w)$ is bounded on the perimeter of the rectangle $\mathbb{O}$ if $B$ and $A_m$
tend to infinity. Thus, it follows that
\begin{eqnarray}\label{sz} \lim\sb{\substack{{B\rightarrow\infty}\\
{{m}\rightarrow\infty}}}\frac{1}{2\pi{i}}\int_{\mathbb{O}}\frac{z}{w(w-z)}\frac{wJ_{\nu+2}(w)-4J_{\nu+1}(w)}{2J_{\nu}(w)-wJ_{\nu+1}(w)}dw
=0.\end{eqnarray}
Now, (\ref{a4}) and (\ref{sz}) imply
$$\frac{\nu(\nu-2)J_\nu(z)+(3-2\nu)zJ_\nu'(z)+z^2J_\nu''(z)}{2(2-\nu)J_\nu(z)+2zJ_\nu'(z)}=
\frac{z}{2}\cdot\frac{zJ_{\nu+2}(z)-4J_{\nu+1}(z)}{2J_{\nu}(z)-zJ_{\nu+1}(z)}=-\sum_{n\geq 1}\frac{z^2}{\beta_{\nu,n}^2-z^2},$$
and finally we get
$$\frac{\nu(\nu-2)J_\nu(z^{\frac{1}{2}})+(3-2\nu)z^{\frac{1}{2}}J_\nu'(z^{\frac{1}{2}})+zJ_\nu''(z^{\frac{1}{2}})}
{2(2-\nu)J_\nu(z^{\frac{1}{2}})+2z^{\frac{1}{2}}J_\nu'(z^{\frac{1}{2}})}=-\sum_{n\geq1}\frac{z}{\beta_{\nu,n}^2-z}.$$
\end{proof}

\begin{lemma}\label{ccccccccccvvvvvvvvvvvv}
If $\nu\geq 0,$ then $\alpha_{\nu,1}>1,$ where $\alpha_{\nu,1}$ denotes the first positive root of the equation $J_{\nu}(x)-xJ_{\nu+1}(x)=0.$ Similarly, if $\nu\geq0,$ then $\beta_{\nu,1}>1,$ where $\beta_{\nu,1}$ denotes the first positive root of the equation $2J_{\nu}(x)-xJ_{\nu+1}(x)=0.$
\end{lemma}

\begin{proof}[\bf Proof]
We shall use the following integral representation \cite[p. 224]{nist}
\begin{equation}\label{integrep}J_{\nu}(x)=\frac{2\left(\frac{x}{2}\right)^{\nu}}{\sqrt{\pi}
\Gamma\left(\nu+\frac{1}{2}\right)}\int_0^1(1-t^2)^{\nu-\frac{1}{2}}\cos(xt)dt,\end{equation}
which is valid for all $x\in\mathbb{R}$ and $\nu>-\frac{1}{2}.$
Observe that for $x\in(0,1]$ and $\nu\geq0$ we have
\begin{align*}J_{\nu}(x)-xJ_{\nu+1}(x)&=\frac{2\left(\frac{x}{2}\right)^{\nu}}{\sqrt{\pi}\Gamma\left(\nu+\frac{1}{2}\right)}\int_0^1\left(1-\frac{x^2}{2\nu+1}+\frac{x^2t^2}{2\nu+1}\right)(1-t^2)^{\nu-\frac{1}{2}}\cos(xt)dt\\
&>\frac{2\left(\frac{x}{2}\right)^{\nu}}{\sqrt{\pi}\Gamma\left(\nu+\frac{1}{2}\right)}\int_0^1\frac{2\nu}{2\nu+1}(1-t^2)^{\nu-\frac{1}{2}}\cos(xt)dt\geq0,\end{align*}
\begin{align*}2J_{\nu}(x)-xJ_{\nu+1}(x)&=\frac{2\left(\frac{x}{2}\right)^{\nu}}{\sqrt{\pi}\Gamma\left(\nu+\frac{1}{2}\right)}\int_0^1\left(2-\frac{x^2}{2\nu+1}+\frac{x^2t^2}{2\nu+1}\right)(1-t^2)^{\nu-\frac{1}{2}}\cos(xt)dt\\
&>\frac{2\left(\frac{x}{2}\right)^{\nu}}{\sqrt{\pi}\Gamma\left(\nu+\frac{1}{2}\right)}\int_0^1\frac{4\nu+1}{2\nu+1}(1-t^2)^{\nu-\frac{1}{2}}\cos(xt)dt>0,\end{align*}
Thus, the smallest positive roots of the transcendent equations in the question, that is, $J_{\nu}(x)-xJ_{\nu+1}(x)=0$ and $2J_{\nu}(x)-xJ_{\nu+1}(x)=0,$ must be bigger then one.
\end{proof}

\begin{lemma}\label{lemlandau}\cite[p. 196]{landau} For $\nu>-1$ let $\gamma_{\nu,n}$ be the $n$th positive root of the equation $\gamma{J}_\nu(z)+zJ_\nu'(z)=0.$ If $\nu+\gamma\geq0,$ then the function $\nu\mapsto \gamma_{\nu,n}$ is strictly increasing on $(-1,\infty)$ for $n\in\{1,2,\dots\}$ fixed.
\end{lemma}

\section{\bf Proofs of the Main Results}
\setcounter{equation}{0}

In this section our aim is to prove the main results of this paper.

\begin{proof}[\bf Proof of Theorem \ref{th1}]
Observe that
$$1+\frac{zf_{\nu}''(z)}{f'_{\nu}(z)}=1+\frac{zJ_\nu''(z)}{J_\nu'(z)}+\left(\frac{1}{\nu}-1\right)\frac{zJ_\nu'(z)}{J_\nu(z)}.$$
Now, recall the following infinite product representations \cite[p. 235]{nist}
$$J_{\nu}(z)=\frac{\left(\frac{1}{2}z\right)^{\nu}}{\Gamma(\nu+1)}\prod_{n\geq 1}\left(1-\frac{z^2}{j_{\nu,n}^2}\right), \ \
J_{\nu}'(z)=\frac{\left(\frac{1}{2}z\right)^{\nu-1}}{2\Gamma(\nu)}\prod_{n\geq 1}\left(1-\frac{z^2}{j_{\nu,n}'^2}\right),$$
where $j_{\nu,n}$ and $j_{\nu,n}'$ are the $n$th positive roots of $J_{\nu}$ and $J_{\nu}',$ respectively. Logarithmic differentiation yields
$$\frac{zJ'_\nu(z)}{J_\nu(z)}=\nu-\sum_{n\geq1}\frac{2z^2}{j_{\nu,n}^2-z^2},\ \ 1+\frac{zJ_\nu''(z)}{J_\nu'(z)}=\nu-\sum_{n\geq1}\frac{2z^2}{j_{\nu,n}'^2-z^2},$$
which implies that
$$1+\frac{zf_{\nu}''(z)}{f_{\nu}'(z)}=1-\left(\frac{1}{\nu}-1\right)\sum_{n\geq1}\frac{2z^2}{j_{\nu,n}^2-z^2}-\sum_{n\geq1}\frac{2z^2}{j_{\nu,n}'^2-z^2}.$$
Now, suppose that $\nu\in(0,1].$ By using the inequality \eqref{kk}, for all $z\in\mathbb{D}(0,j_{\nu,1}')$ we obtain the inequality
\begin{equation}\label{fnuconvex}\real\left(1+\frac{zf_{\nu}''(z)}{f_{\nu}'(z)}\right)\geq1-\left(\frac{1}{\nu}-1\right)
\sum_{n\geq1}\frac{2r^2}{j_{\nu,n}^2-r^2}-\sum_{n\geq1}\frac{2r^2}{j_{\nu,n}'^2-r^2},\end{equation}
where $|z|=r.$ Moreover, observe that if we use the inequality \eqref{llll} then we get that the above inequality is also valid when $\nu>1.$ Here we used that the zeros $j_{\nu,n}$ and $j_{\nu,n}'$ interlace according to the inequalities \cite[p. 235]{nist} \begin{equation}\label{interlace}\nu\leq j_{\nu,1}'<j_{\nu,1}<j_{\nu,2}'<j_{\nu,2}<j_{\nu,3}'<{\dots}.\end{equation} Now, the above deduced inequality implies for $r\in(0,j_{\nu,1}')$
$$\inf_{z\in\mathbb{D}(0,r)}\left\{\real\left(1+\frac{zf_{\nu}''(z)}{f_{\nu}'(z)}\right)\right\}=1+\frac{rf_{\nu}''(r)}{f_{\nu}'(r)}.$$
On the other hand, the function $u_{\nu}:(0,j_{\nu,1}')\to\mathbb{R},$ defined by $$u_{\nu}(r)=1+\frac{rf_{\nu}''(r)}{f_{\nu}'(r)},$$ is  strictly decreasing
since \begin{align*}u_{\nu}'(r)&=-\left(\frac{1}{\nu}-1\right)\sum_{n\geq1}\frac{4rj_{\nu,n}^2}{(j_{\nu,n}^2-r^2)^2}
-\sum_{n\geq1}\frac{4rj_{\nu,n}'^2}{(j_{\nu,n}'^2-r^2)^2}\\&<\sum_{n\geq1}\frac{4rj_{\nu,n}^2}{(j_{\nu,n}^2-r^2)^2}
-\sum_{n\geq1}\frac{4rj_{\nu,n}'^2}{(j_{\nu,n}'^2-r^2)^2}<0\end{align*}
for $\nu>0$ and $r\in(0,j_{\nu,1}').$ Here we used again that the zeros $j_{\nu,n}$ and $j_{\nu,n}'$ interlace and for all $n\in\{1,2,\dots\},$ $\nu>0$ and $r<\sqrt{j_{\nu,1}j_{\nu,1}'}$ we have that $$j_{\nu,n}^2(j_{\nu,n}'^2-r^2)^2<j_{\nu,n}'^2(j_{\nu,n}^2-r^2)^2.$$ Observe also that $\lim_{r\searrow0}u_{\nu}(r)=1>\alpha$ and $\lim_{r\nearrow j_{\nu,1}'}u_{\nu}(r)=-\infty,$ which means that for $z\in\mathbb{D}(0,r_1)$ we have
$$\real\left(1+\frac{zf_{\nu}''(z)}{f_{\nu}'(z)}\right)>\alpha,$$
if and only if  $r_1$ is the unique root of
$$1+\frac{rf_{\nu}''(r)}{f_{\nu}'(r)}=\alpha,$$
situated in $(0,j_{\nu,1}').$
\end{proof}

\begin{proof}[\bf Proof of Theorem \ref{th2}]
Observe that
$$z\frac{g_{\nu}''(z)}{g_{\nu}'(z)}=\frac{\nu(\nu-1)J_\nu(z)+2(1-\nu)zJ_\nu'(z)+z^2J_\nu''(z)}{(1-\nu)J_\nu(z)+zJ_\nu'(z)}.$$
The recurrence formula \cite[p. 222]{nist} $zJ'_\nu(z)=\nu{J_\nu(z)}-zJ_{\nu+1}(z)$
implies
$$z\frac{g''_{\nu}(z)}{g'_{\nu}(z)}=z\frac{zJ_{\nu+2}(z)-3J_{\nu+1}(z)}{J_{\nu}(z)-zJ_{\nu+1}(z)},$$
and using (\ref{sxv}) it follows that
$$1+z\frac{g''_{\nu}(z)}{g'_{\nu}(z)}=1-\sum_{n\geq1}\frac{2z^2}{\alpha_{\nu,n}^2-z^2}.$$
Application of the inequality \eqref{kk} implies that
$$\real\left(1+z\frac{g''_{\nu}(z)}{g'_{\nu}(z)}\right)\geq1-\sum_{n\geq1}\frac{2r^2}{\alpha_{\nu,n}^2-r^2},$$
where $|z|=r.$ Thus, for $r\in(0,\alpha_{\nu,1})$ we get
$$\inf_{z\in\mathbb{D}(0,r)}\left\{\real\left(1+\frac{zg_{\nu}''(z)}{g_{\nu}'(z)}\right)\right\}=
1-\sum_{n\geq1}\frac{2r^2}{\alpha_{\nu,n}^2-r^2}=1+\frac{rg_{\nu}''(r)}{g_{\nu}'(r)}.$$
The function
$v_{\nu}:(0,\alpha_{\nu,1})\to\mathbb{R},$ defined by
$$v_{\nu}(r)=1+\frac{rg_{\nu}''(r)}{g_{\nu}'(r)},$$ is strictly decreasing
and $$\lim_{r\searrow0}v_{\nu}(r)=1, \ \ \ \lim_{r\nearrow\alpha_{\nu,1}}v_{\nu}(r)=-\infty.$$
Consequently, the equation
$$1+\frac{rg_{\nu}''(r)}{g_{\nu}'(r)}=\alpha$$  has a unique root
$r_2$ in $(0,\alpha_{\nu,1}),$ and this equation is equivalent to
$$ 1+r\frac{rJ_{\nu+2}(r)-3J_{\nu+1}(r)}{J_{\nu}(r)-rJ_{\nu+1}(r)}=\alpha.$$ In other words, we have
$$\real\left(1+\frac{zg_{\nu}''(z)}{g_{\nu}'(z)}\right)>\alpha, \
z\in{\mathbb{D}(0,r_2)} \ \ \textrm{and} \ \
\inf_{z\in{\mathbb{D}(0,r_2)}}\left\{\real\left(1+\frac{zg_{\nu}''(z)}{g_{\nu}'(z)}\right)\right\}=\alpha.$$
Finally, let us recall that (see \cite[p. 597]{watson}) when $\rho+\nu>0$
and $\nu>-1$ the so-called Dini function $z\mapsto
zJ_{\nu}'(z)+\rho J_{\nu}(z)$ has only real zeros and according to
Ismail and Muldoon \cite[p. 11]{ismail} we know that the smallest
positive zero of the above function is less than $j_{\nu,1}.$ This in turn implies
that $\alpha_{\nu,1}<j_{\nu,1},$ which completes the proof.
\end{proof}

\begin{proof}[\bf Proof of Theorem \ref{th3}]
Observe that
$$z\frac{h_{\nu}''(z)}{h_{\nu}'(z)}=\frac{\nu(\nu-2)J_\nu(z^{\frac{1}{2}})+(3-2\nu)z^{\frac{1}{2}}J_\nu'(z^{\frac{1}{2}})+zJ_\nu''(z^{\frac{1}{2}})}{2(2-\nu)J_\nu(z^{\frac{1}{2}})+2z^{\frac{1}{2}}J_\nu'(z^{\frac{1}{2}})},$$
and according to Lemma \ref{lem5} it follows
$$1+z\frac{h_{\nu}''(z)}{h_{\nu}'(z)}=1-\sum_{n\geq1}\frac{z}{\beta_{\nu,n}^2-z}.$$
Let $r\in\left(0,\beta_{\nu,1}^2\right)$ be a fixed number. The minimum principle for harmonic functions and inequality \eqref{llll} for $\lambda=0$ imply that for $z\in\mathbb{D}(0,r)$ we have
\begin{align*}
\real&\left(1+z\frac{h_{\nu}''(z)}{h_{\nu}'(z)}\right)=
\real\left(1-\sum_{n\geq1}\frac{z}{\beta_{\nu,n}^2-z}\right)\geq
\min_{|z|=r}\real\left(1-\sum_{n\geq1}\frac{z}{\beta_{\nu,n}^2-z}\right)\\
&=\min_{|z|=r}\left(1-\sum_{n\geq1}\real\frac{z}{\beta_{\nu,n}^2-z}\right)
\geq1-\sum_{n\geq1}\frac{r}{\beta_{\nu,n}^2-r}=1+r\frac{h_{\nu}''(r)}{h_{\nu}'(r)}.
\end{align*}
Consequently, it follows that
$$\inf_{z\in{\mathbb{D}(0,r)}}\left\{\real\left(1+z\frac{h_{\nu}''(z)}{h_{\nu}'(z)}\right)\right\}=
1+r\frac{h_{\nu}''(r)}{h_{\nu}'(r)}=1+\frac{r^{\frac{1}{2}}}{2}\cdot\frac{r^{\frac{1}{2}}J_{\nu+2}(r^{\frac{1}{2}})-4J_{\nu+1}(r^{\frac{1}{2}})}
{2J_{\nu}(r^{\frac{1}{2}})-r^{\frac{1}{2}}J_{\nu+1}(r^{\frac{1}{2}})}.$$
Now, let $r_3$  be the smallest positive  root of the equation
\begin{eqnarray}\label{w14s2d5c1ssss6a2j333}\alpha=1+\frac{r^{\frac{1}{2}}}{2}\cdot\frac{r^{\frac{1}{2}}J_{\nu+2}(r^{\frac{1}{2}})-4J_{\nu+1}(r^{\frac{1}{2}})}
{2J_{\nu}(r^{\frac{1}{2}})-r^{\frac{1}{2}}J_{\nu+1}(r^{\frac{1}{2}})}.\end{eqnarray}
For $z\in\mathbb{D}(0,r_3)$ we have
$$\real\left(1+z\frac{h_{\nu}''(z)}{h_{\nu}'(z)}\right)>\alpha.$$ In order to finish the proof, we need to show that the equation (\ref{w14s2d5c1ssss6a2j333}) has a unique root in $\left(0,\beta_{\nu,1}^2\right).$
But, the equation (\ref{w14s2d5c1ssss6a2j333}) is equivalent to
$$w_{\nu}(r)=1-\alpha-\sum_{n\geq1}\frac{r}{\beta_{\nu,n}^2-r}=0,$$
and we have $$\lim_{r\searrow0}w_{\nu}(r)=1-\alpha>0, \   \   \  \lim_{r\nearrow\beta_{\nu,1}^2}w_{\nu}(r)=-\infty.$$
Now, since the function $w_{\nu}$ is strictly decreasing on $(0,\beta_{\nu,1}^2),$ it follows that the equation $w_{\nu}(r)=0$ has a unique root.

Finally, since for $\rho+\nu>0$
and $\nu>-1$ the function $z\mapsto
zJ_{\nu}'(z)+\rho J_{\nu}(z)$ has only real zeros (see \cite[p. 597]{watson}) and the smallest
positive zero of the above function is less than $j_{\nu,1}$ (see \cite[p. 11]{ismail}), we obtain that $\beta_{\nu,1}<j_{\nu,1},$ which completes the proof.
\end{proof}

\begin{proof}[\bf Proof of Theorem \ref{th4}]
According to \eqref{fnuconvex} for $z\in\mathbb{D}$ we obtain that
\begin{align*}\real\left(1+\frac{zf_{\nu}''(z)}{f_{\nu}'(z)}\right)&\geq1-\left(\frac{1}{\nu}-1\right)
\sum_{n\geq1}\frac{2r^2}{j_{\nu,n}^2-r^2}-\sum_{n\geq1}\frac{2r^2}{j_{\nu,n}'^2-r^2}\\
&\geq1-\left(\frac{1}{\nu}-1\right)
\sum_{n\geq1}\frac{2}{j_{\nu,n}^2-1}-\sum_{n\geq1}\frac{2}{j_{\nu,n}'^2-1}\\
&=1+\frac{J_\nu''(1)}{J_\nu'(1)}+\left(\frac{1}{\nu}-1\right)\frac{J_\nu'(1)}{J_\nu(1)}=1+\frac{f_{\nu}''(1)}{f_{\nu}'(1)}.\end{align*}
Now, consider the function $u:(\nu^{\ast},\infty)\to \mathbb{R},$ defined by
$$u(\nu)=1+\frac{J_\nu''(1)}{J_\nu'(1)}+\left(\frac{1}{\nu}-1\right)\frac{J_\nu'(1)}{J_\nu(1)}=1-\left(\frac{1}{\nu}-1\right)\sum_{n\geq1}\frac{2}{j_{\nu,n}^2-1}-\sum_{n\geq1}\frac{2}{j_{\nu,n}'^2-1}.$$
We note that this function is well defined since $J_{\nu}(1)>0$ and $J_{\nu}'(1)>0$ when $\nu>\nu^{\ast}.$ By using \eqref{integrep} for $\nu>-\frac{1}{2}$ we get \begin{equation}\label{inBesNeu}J_{\nu}(1)=\frac{2\left(\frac{1}{2}\right)^{\nu}}{\sqrt{\pi}
\Gamma\left(\nu+\frac{1}{2}\right)}\int_0^1(1-t^2)^{\nu-\frac{1}{2}}\cos(t)dt>0.\end{equation}
Moreover, since $\nu\mapsto j_{\nu,n}$ is strictly increasing on $(0,\infty)$ for each $n\in\{1,2,\dots\}$ (see \cite[p. 236]{nist}), it follows that
$$\frac{d}{d\nu}\left(\frac{J_{\nu}'(1)}{J_{\nu}(1)}\right)=
\frac{d}{d\nu}\left(\nu-\sum_{n\geq1}\frac{2}{j_{\nu,n}^2-1}\right)=1+\sum_{n\geq1}\frac{4j_{\nu,n}\frac{d j_{\nu,n}}{d\nu}}{(j_{\nu,n}^2-1)^2}>0$$
for $\nu>0.$ This means that if $\nu>\nu^{\ast},$ then $J_{\nu}'(1)/J_{\nu}(1)>J_{\nu^{\ast}}'(1)/J_{\nu^{\ast}}(1)=0.$

Now, in what follows we show that $u$ is strictly increasing. For this we distinguish two cases. First we consider
that $\nu\in(\nu^{\ast},1].$ Since the functions $\nu\mapsto j_{\nu,n}$ and $\nu\mapsto j_{\nu,n}'$ are strictly increasing on $[0,\infty)$ for each $n\in\{1,2,\dots\}$ (see \cite[p. 236]{nist}), it follows that the functions $\nu\mapsto 2/(j_{\nu,n}^2-1)$ and $\nu\mapsto 2/(j_{\nu,n}'^2-1)$ are strictly decreasing on $[0,\infty)$ for each $n\in\{1,2,\dots\},$ and consequently $u$ is strictly increasing on $(\nu^{\ast},1].$

Suppose that $\nu>1.$ In this case we have
$$u'(\nu)=\frac{1}{\nu^2}\sum_{n\geq1}\frac{2}{j_{\nu,n}^2-1}-4\left(1-\frac{1}{\nu}\right)\sum_{n\geq1}\frac{j_{\nu,n}\frac{d j_{\nu,n}}{d \nu}}{(j_{\nu,n}^2-1)^2}+4\sum_{n\geq1}\frac{j_{\nu,n}'\frac{d j_{\nu,n}'}{d \nu}}{(j_{\nu,n}'^2-1)^2}.$$
Recall that for any $n\in\{1,2,\dots\}$ the derivative of $j_{\nu,n}$ and $j_{\nu,n}'$ with respect to $\nu$ can be written as \cite[p. 236]{nist}
\begin{equation}\label{gyokder}\frac{d j_{\nu,n}}{d \nu}=2j_{\nu,n}\int_0^{\infty}K_0(2j_{\nu,n}\sinh(t))e^{-2\nu t}dt,\end{equation}
$$\frac{d j_{\nu,n}'}{d \nu}=\frac{2j_{\nu,n}'}{j_{\nu,n}'^2-\nu^2}\int_0^{\infty}(j_{\nu,n}'^2\cosh(2t)-\nu^2)K_0(2j_{\nu,n}'\sinh(t))e^{-2\nu t}dt,$$
where $K_0$ stands for the modified Bessel function of the second kind and zero order (see \cite[p. 252]{nist}). Observe that for $\nu>1,$ $n\in\{1,2,\dots\}$ and $t>0$ we have
\begin{equation}\label{qwe1}\frac{j_{\nu,n}'^2\cosh(2t)-\nu^2}{j_{\nu,n}'^2-\nu^2}>1>1-\frac{1}{\nu},\end{equation}
\begin{equation}\label{qwe2}\frac{j_{\nu,n}'^2}{(j_{\nu,n}'^2-1)^2}>\frac{j_{\nu,n}^2}{(j_{\nu,n}^2-1)^2},\end{equation}
since according to \eqref{interlace} we have that $j_{\nu,n}^2>j_{\nu,n}'^2$ and $j_{\nu,n}^2j_{\nu,n}'^2>j_{\nu,1}^2j_{\nu,1}'^2>j_{1,1}^2j_{1,1}'^2>1.$ On the other hand, we know that $K_0$ is strictly decreasing on $(0,\infty)$, and this implies that for each $\nu>1,$ $n\in\{1,2,\dots\}$ and $t>0$ we have
$$K_0(2j_{\nu,n}'\sinh(t))>K_0(2j_{\nu,n}\sinh(t)).$$ Combining this with \eqref{qwe1} and \eqref{qwe2} we obtain
\begin{align*}
4\sum_{n\geq1}\frac{j_{\nu,n}'\frac{d j_{\nu,n}'}{d \nu}}{(j_{\nu,n}'^2-1)^2}&=
8\sum_{n\geq1}\int_0^{\infty}\frac{j_{\nu,n}'^2\cosh(2t)-\nu^2}{j_{\nu,n}'^2-\nu^2}\frac{j_{\nu,n}'^2}{(j_{\nu,n}'^2-1)^2}K_0(2j_{\nu,n}'\sinh(t))e^{-2\nu t}dt\\
&>8\sum_{n\geq1}\int_0^{\infty}\frac{j_{\nu,n}'^2\cosh(2t)-\nu^2}{j_{\nu,n}'^2-\nu^2}\frac{j_{\nu,n}'^2}{(j_{\nu,n}'^2-1)^2}K_0(2j_{\nu,n}\sinh(t))e^{-2\nu t}dt\\
&>8\sum_{n\geq1}\int_0^{\infty}\left(1-\frac{1}{\nu}\right)\frac{j_{\nu,n}^2}{(j_{\nu,n}^2-1)^2}K_0(2j_{\nu,n}\sinh(t))e^{-2\nu t}dt\\
&=4\left(1-\frac{1}{\nu}\right)\sum_{n\geq1}\frac{j_{\nu,n}\frac{d j_{\nu,n}}{d \nu}}{(j_{\nu,n}^2-1)^2},
\end{align*}
which implies that $u'(\nu)>0$ for $\nu>1,$ and thus the function $u$ is strictly increasing on $(1,\infty),$ and hence on the whole $(\nu^{\ast},\infty).$ Consequently, if $\nu\geq\nu_{\alpha}=\nu_{\alpha}(f_{\nu}),$ then we get the inequality $u(\nu)\geq u(\nu_{\alpha}).$ This in turn implies that $\nu_{\alpha}$ is the smallest value having the property that the condition $\nu\geq\nu_{\alpha}$ implies that for all $z\in\mathbb{D}$ we have
$$\real\left(1+z\frac{f_{\nu}''(z)}{f_{\nu}'(z)}\right)>u(\nu_{\alpha})=\alpha.$$ Thus, we proved that the function $f_{\nu}$ is convex of order $\alpha\in[0,1)$ in $\mathbb{D}$ if and only if
$\nu\geq\nu_{\alpha}(f_{\nu}),$ where $\nu_{\alpha}=\nu_{\alpha}(f_{\nu})$ is the unique  root of
the equation $u(\nu)=\alpha,$ that is,
$$1+\frac{J_\nu''(1)}{J_\nu'(1)}+\left(\frac{1}{\nu}-1\right)\frac{J_\nu'(1)}{J_\nu(1)}=\alpha.$$
Since $J_{\nu}$ is a particular solution of the Bessel differential equation (see \cite[p. 217]{nist}), it follows that
$$J_{\nu}''(1)+J_{\nu}'(1)+(1-\nu^2)J_{\nu}(1)=0,$$
and by using this the above equation can be rewritten as
$$\nu(\nu^2-1)J_{\nu}^2(1)+(1-\nu)(J_{\nu}'(1))^2=\alpha\nu J_{\nu}(1)J_{\nu}'(1).$$
In particular, when $\alpha=0$ the above equation becomes
$(\nu-1)(\nu(\nu+1)J_{\nu}^2(1)-(J_{\nu}'(1))^2)=0,$ and by using the recurrence relation \cite[p. 222]{nist} \begin{equation}\label{rec2}zJ_{\nu}'(z)=zJ_{\nu-1}(z)-\nu J_{\nu}(z),\end{equation} we obtain that the above equation can be rewritten as
$$(\nu-1)(\nu J_{\nu}^2(1)+2\nu J_{\nu-1}(1)J_{\nu}(1)-J_{\nu-1}^2(1))=0.$$
This equation has two roots: $\nu_1=1$ and $\nu_2\simeq0.1246{\dots},$ however, only $\nu_1$ is situated in $(\nu^{\ast},\infty).$
\end{proof}

\begin{proof}[\bf Proof of Theorem \ref{th5}]
Let $\nu\geq0.$ According to Lemma \ref{ccccccccccvvvvvvvvvvvv} we have $\alpha_{\nu,1}>1.$  Taking into account the proof of Theorem \ref{th2} when $r=1$ we get for $z\in\mathbb{D}$
\begin{align*}
\real&\left(1+z\frac{g_{\nu}''(z)}{g_{\nu}'(z)}\right)=\real\left(1-\sum_{n\geq1}\frac{2z^2}{\alpha_{\nu,n}^2-z^2}\right)
\geq\min_{|z|=1}\real\left(1-\sum_{n\geq1}\frac{2z^2}{\alpha_{\nu,n}^2-z^2}\right)\\&\geq1-\sum_{n\geq1}\frac{2}{\alpha_{\nu,n}^2-1}
=1+\frac{g_{\nu}''(1)}{g_{\nu}'(1)}=1+\frac{J_{\nu+2}(1)-3J_{\nu+1}(1)}{J_{\nu}(1)-J_{\nu+1}(1)}.
\end{align*}
According to Lemma \ref{lemlandau}, the function $\nu\rightarrow\alpha_{\nu,n}$ is strictly increasing on $[0,\infty)$ for every fixed natural number $n.$
Thus, it follows that the function $v:[0,\infty)\rightarrow\mathbb{R},$ defined by
$$v(\nu)=1+\frac{J_{\nu+2}(1)-3J_{\nu+1}(1)}{J_{\nu}(1)-J_{\nu+1}(1)}=1-\sum_{n\geq1}\frac{2}{\alpha_{\nu,n}^2-1},$$
is strictly increasing too. Here we used that $J_{\nu}(1)-J_{\nu+1}(1)\neq0$ when $\nu\geq0,$ since from \eqref{integrep} we have
$$J_{\nu}(1)-J_{\nu+1}(1)=\frac{2\left(\frac{1}{2}\right)^{\nu}}{\sqrt{\pi}\Gamma\left(\nu+\frac{1}{2}\right)}\int_0^1
\left(1-\frac{1}{2\nu+1}+\frac{t^2}{2\nu+1}\right)(1-t^2)^{\nu-\frac{1}{2}}\cos(t)dt>0.$$
Since the function $v$ is strictly increasing, it follows that if $\nu\geq\nu_{\alpha}=\nu_{\alpha}(g_{\nu}),$ then we get the inequality
\begin{equation}\label{s25zf99999999}
1+\frac{J_{\nu+2}(1)-3J_{\nu+1}(1)}{J_{\nu}(1)-J_{\nu+1}(1)}=v(\nu)\geq{v}(\nu_{\alpha})=1+\frac{J_{\nu_{\alpha}+2}(1)-3J_{\nu_{\alpha}+1}(1)}
{J_{\nu_{\alpha}}(1)-J_{\nu_{\alpha}+1}(1)}=\alpha.
\end{equation}
Now, from (\ref{s25zf99999999}) we get that $\nu_{\alpha}$ is the smallest value having the property that the condition $\nu\geq\nu_{\alpha}$ implies
$$\real\left(1+z\frac{g_{\nu}''(z)}{g_{\nu}'(z)}\right)>1+\frac{J_{\nu_{\alpha}+2}(1)-3J_{\nu_{\alpha}+1}(1)}{J_{\nu_{\alpha}}(1)-J_{\nu_{\alpha}+1}(1)}=\alpha \ \ \ \mbox{for all}\ \ \ z\in\mathbb{D}.$$
Thus, we proved that the function $g_{\nu}$ is convex of order $\alpha\in[0,1)$ in $\mathbb{D}$ if and only if
$\nu\geq\nu_{\alpha}(g_{\nu}),$ where $\nu_{\alpha}=\nu_{\alpha}(g_{\nu})$ is the unique  root of
the equation
$$ 1+\frac{J_{\nu+2}(1)-3J_{\nu+1}(1)}{J_{\nu}(1)-J_{\nu+1}(1)}=\alpha$$
situated in $[0,\infty).$ Observe that by using the recurrence relation \cite[p. 222]{nist}
\begin{equation}\label{recBes}J_{\nu}(z)+J_{\nu+2}(z)=\frac{2(\nu+1)}{z}J_{\nu+1}(z)\end{equation} we get that the above equation is equivalent to
$(2\nu+\alpha-2)J_{\nu+1}(1)=\alpha J_{\nu}(1).$
In particular, $g_{\nu}$ is convex if and only if
$\nu\geq\nu_0=\nu_0(g_{\nu}),$ where $\nu_0=1$ is the unique root of the equation
$J_{\nu}(1)-4J_{\nu+1}(1)+J_{\nu+2}(1)=0,$ that is, $(2\nu-2)J_{\nu+1}(1)=0.$ Here we used that $J_{\nu+1}(1)>0$ for $\nu>-\frac{3}{2},$ which follows from \eqref{inBesNeu}.
\end{proof}

\begin{proof}[\bf Proof of Theorem \ref{th6}]
Let $\nu\geq0.$ According to Lemma \ref{ccccccccccvvvvvvvvvvvv} we have $\beta_{\nu,1}>1.$  Taking into account the proof of Theorem \ref{th3} when $r=1$ we get for $z\in\mathbb{D}$
\begin{align*}
\real&\left(1+z\frac{h_{\nu}''(z)}{h_{\nu}'(z)}\right)=\real\left(1-\sum_{n\geq1}\frac{z}{\beta_{\nu,n}^2-z}\right)
\geq\min_{|z|=1}\real\left(1-\sum_{n\geq1}\frac{z}{\beta_{\nu,n}^2-z}\right)\\&\geq1-\sum_{n\geq1}\frac{1}{\beta_{\nu,n}^2-1}
=1+\frac{h_{\nu}''(1)}{h_{\nu}'(1)}=1+\frac{1}{2}\cdot\frac{J_{\nu+2}(1)-4J_{\nu+1}(1)}{2J_{\nu}(1)-J_{\nu+1}(1)}.
\end{align*}
According to Lemma \ref{lemlandau}, the function $\nu\rightarrow\beta_{\nu,n}$ is strictly increasing on $(-1,\infty)$ for every fixed natural number $n.$
Thus, it follows that the function $w:[0,\infty)\rightarrow\mathbb{R},$ defined by
$$w(\nu)=1+\frac{1}{2}\cdot\frac{J_{\nu+2}(1)-4J_{\nu+1}(1)}{2J_{\nu}(1)-J_{\nu+1}(1)}=1-\sum_{n\geq1}\frac{1}{\beta_{\nu,n}^2-1},$$
is strictly increasing too. Note that the function $w$ is well defined since from \eqref{integrep} for $\nu\geq0$ we have
$$2J_{\nu}(1)-J_{\nu+1}(1)=\frac{2\left(\frac{1}{2}\right)^{\nu}}{\sqrt{\pi}\Gamma\left(\nu+\frac{1}{2}\right)}
\int_0^1\left(2-\frac{1}{2\nu+1}+\frac{t^2}{2\nu+1}\right)(1-t^2)^{\nu-\frac{1}{2}}\cos(t)dt>0.$$
Now, since $w$ is strictly increasing, it follows that if $\nu\geq\nu_{\alpha}=\nu_{\alpha}(h_{\nu}),$ then we get the inequality
\begin{equation}\label{s25zf99999999ss}
1+\frac{1}{2}\cdot\frac{J_{\nu+2}(1)-4J_{\nu+1}(1)}{2J_{\nu}(1)-J_{\nu+1}(1)}=w(\nu)\geq{w}(\nu_{\alpha})=1+\frac{1}{2}\cdot\frac{J_{\nu_{\alpha}+2}(1)-4J_{\nu_{\alpha}+1}(1)}
{2J_{\nu_{\alpha}}(1)-J_{\nu_{\alpha}+1}(1)}=\alpha.
\end{equation}
Now, from (\ref{s25zf99999999ss}) we get that $\nu_{\alpha}$ is the smallest value having the property that the condition $\nu\geq\nu_{\alpha}$ implies
$$\real\left(1+z\frac{h_{\nu}''(z)}{h_{\nu}'(z)}\right)>1+\frac{1}{2}\cdot\frac{J_{\nu_{\alpha}+2}(1)-4J_{\nu_{\alpha}+1}(1)}{2J_{\nu_{\alpha}}(1)-J_{\nu_{\alpha}+1}(1)}=\alpha \ \ \ \mbox{for all}\ \ \ z\in\mathbb{D}.$$

Summarizing, we proved that the function $h_{\nu}$ is convex of order $\alpha\in[0,1)$ in $\mathbb{D}$ if and only if
$\nu\geq\nu_{\alpha}(h_{\nu}),$ where $\nu_{\alpha}=\nu_{\alpha}(h_{\nu})$ is the unique  root of
the equation
$$
 1+\frac{1}{2}\cdot\frac{J_{\nu+2}(1)-4J_{\nu+1}(1)}{2J_{\nu}(1)-J_{\nu+1}(1)}=\alpha
$$
situated in $[0,\infty).$ By using \eqref{recBes} the above equation can be rewritten as
$$(2\nu+2\alpha-4)J_{\nu+1}(1)=(4\alpha-3)J_{\nu}(1).$$ In particular, $h_{\nu}$ is convex if and only if
$\nu\geq\nu_0=\nu_0(h_{\nu}),$ where $\nu_0\simeq0.6688\dots$ is the unique root of the equation
$4J_{\nu}(1)-6J_{\nu+1}(1)+J_{\nu+2}(1)=0,$ that is, $2(\nu-2)J_{\nu+1}(1)+3J_{\nu}(1)=0.$

Finally, observe that, in particular, $h_{\nu}$ is convex of order $\frac{3}{4}$ if and only if
$\nu\geq\nu_{\frac{3}{4}}(h_{\nu}),$ where $\nu_{\frac{3}{4}}(h_{\nu})=\frac{5}{4}$ is the unique root of the equation
$(4\nu-5)J_{\nu+1}(1)=0.$ Here we used again that according to \eqref{inBesNeu} we have $J_{\nu+1}(1)>0$ for $\nu>-\frac{3}{2}.$
\end{proof}

\begin{proof}[\bf Proof of Theorem \ref{th7}]
Observe that
$$z\frac{\varphi_{\nu}''(z)}{\varphi_{\nu}'(z)}=
\frac{\nu(\nu+2)J_\nu(z^{\frac{1}{2}})-(2\nu+1)z^{\frac{1}{2}}J_\nu'(z^{\frac{1}{2}})+zJ_\nu''(z^{\frac{1}{2}})}
{-2\nu J_\nu(z^{\frac{1}{2}})+2z^{\frac{1}{2}}J_\nu'(z^{\frac{1}{2}})}.$$
Now, by using the recurrence relation \eqref{rec2} and the fact that $J_{\nu}$ satisfies \cite[p. 217]{nist}
$$z^2J_{\nu}''(z)+zJ_{\nu}'(z)+(z^2-\nu^2)J_{\nu}(z)=0,$$
we obtain that
$$1+z\frac{\varphi_{\nu}''(z)}{\varphi_{\nu}'(z)}=\frac{z^{\frac{1}{2}}J_{\nu}(z^{\frac{1}{2}})}{2J_{\nu+1}(z^{\frac{1}{2}})}-\nu=
\frac{1}{2}\left[\frac{z^{\frac{1}{2}}J_{\nu+1}'(z^{\frac{1}{2}})}{J_{\nu+1}(z^{\frac{1}{2}})}-(\nu-1)\right].$$
Thus, in view of the Mittag-Leffler expansion
$$\frac{zJ'_\nu(z)}{J_\nu(z)}=\nu-\sum_{n\geq1}\frac{2z^2}{j_{\nu,n}^2-z^2}$$
one has
$$1+z\frac{\varphi_{\nu}''(z)}{\varphi_{\nu}'(z)}=1-\sum_{n\geq1}\frac{z}{j_{\nu+1,n}^2-z}.$$
Now, by using \eqref{kk} it follows that for $z\in\mathbb{D}(0,j_{\nu+1,1})$ we have
\begin{equation}\label{varco}\real\left(1+z\frac{\varphi_{\nu}''(z)}{\varphi_{\nu}'(z)}\right)\geq 1-\sum_{n\geq1}\frac{r}{j_{\nu+1,n}^2-r}=1+r\frac{\varphi_{\nu}''(r)}{\varphi_{\nu}'(r)},\end{equation}
where $r=|z|.$ This inequality implies for $r\in(0,j_{\nu+1,1})$
$$\inf_{z\in\mathbb{D}(0,r)}\left\{\real\left(1+\frac{z\varphi_{\nu}''(z)}{\varphi_{\nu}'(z)}\right)\right\}=1+\frac{r\varphi_{\nu}''(r)}{\varphi_{\nu}'(r)}.$$
On the other hand, the function $\psi_{\nu}:(0,j_{\nu+1,1})\to\mathbb{R},$ defined by $$\psi_{\nu}(r)=1+\frac{r\varphi_{\nu}''(r)}{\varphi_{\nu}'(r)},$$ is  strictly decreasing
since $$\psi_{\nu}'(r)=-\sum_{n\geq1}\frac{j_{\nu+1,n}^2}{(j_{\nu+1,n}^2-r)^2}.$$ Observe also that $\lim_{r\searrow0}\psi_{\nu}(r)=1>\alpha$ and $\lim_{r\nearrow j_{\nu+1,1}}\psi_{\nu}(r)=-\infty,$ which means that for $z\in\mathbb{D}(0,r_4)$ we have
$$\real\left(1+\frac{z\varphi_{\nu}''(z)}{\varphi_{\nu}'(z)}\right)>\alpha,$$
if and only if $r_4$ is the unique root of
$$1+\frac{r\varphi_{\nu}''(r)}{\varphi_{\nu}'(r)}=\alpha,$$
situated in $(0,j_{\nu+1,1}).$
\end{proof}

\begin{proof}[\bf Proof of Theorem \ref{th8}]
First we prove that $J_{\nu+1}(1)>0$ if $\nu>\nu^{\star}.$ For this we show that $\nu\mapsto \log(J_{\nu}(1))$ is increasing on $(\nu^{\star}+1,0).$ This implies that $\nu\mapsto \log(J_{\nu+1}(1))$ is increasing on $(\nu^{\star},-1),$ and consequently $\nu\mapsto J_{\nu+1}(1)$ is increasing too on $(\nu^{\star},-1).$ Thus, $J_{\nu+1}(1)>J_{\nu^{\star}+1}(1)=0$ if $-1>\nu>\nu^{\star}.$ Combining this with \eqref{inBesNeu} we obtain that indeed $J_{\nu+1}(1)>0$ if $\nu>\nu^{\star}.$

By using the infinite product representation \cite[p. 235]{nist}
$$J_{\nu}(1)=\frac{\left(\frac{1}{2}\right)^{\nu}}{\Gamma(\nu+1)}\prod_{n\geq 1}\left(1-\frac{1}{j_{\nu,n}^2}\right),$$
it follows that
$$\Omega(\nu)=\frac{d \log J_{\nu}(1)}{d\nu}=-\log 2-\psi(\nu+1)+\sum_{n\geq 1}\frac{2\frac{d j_{\nu,n}}{d\nu}}{j_{\nu,n}(j_{\nu,n}^2-1)},$$
where $\psi(x)=\Gamma'(x)/\Gamma(x)$ stands for the digamma function, that is, for the logarithmic derivative of the Euler gamma function. We note that the function $\Omega$ is well defined since for $\nu>\nu^{\star}+1$ we have $j_{\nu,n}\neq1$ for each $n\in\{1,2,\dots\}.$ This is because $1=j_{\nu^{\star}+1,1}<j_{\nu,1}<j_{\nu,2}<{\dots}<j_{\nu,n}<{\dots}$ for each $\nu>\nu^{\star}+1$ and $n\in\{1,2,\dots\}.$ Now, by using the known inequalities \cite[p. 196]{isma}
$$\frac{d j_{\nu,n}}{d\nu}>\frac{2}{j_{\nu,n}}+\frac{8(\nu+1)^2}{j_{\nu,n}^3}, \ \ \nu>-1, \  \ n\in\{1,2,\dots\},$$
and \cite[p. 374]{alzer}
$$\psi(x)<\log x-\frac{1}{2x}, \ \ x>0,$$
we obtain
\begin{align*}\Omega(\nu)&>-\log 2-\psi(\nu+1)+\sum_{n\geq 1}\frac{2\frac{d j_{\nu,n}}{d\nu}}{j_{\nu,n}(j_{\nu,n}^2-1)}\\
&>-\log 2-\psi(\nu+1)+\sum_{n\geq 1}\frac{2\frac{d j_{\nu,n}}{d\nu}}{j_{\nu,n}^3}\\
&>-\log 2-\psi(\nu+1)+4\sum_{n\geq 1}\frac{1}{j_{\nu,n}^4}+16(\nu+1)^2\sum_{n\geq1}\frac{1}{j_{\nu,n}^6}\\
&=-\log 2-\psi(\nu+1)+\frac{3\nu+5}{4(\nu+1)^2(\nu+2)(\nu+3)}\\
&>-\log 2-\log(\nu+1)+\frac{1}{2(\nu+1)}+\frac{3\nu+5}{4(\nu+1)^2(\nu+2)(\nu+3)}>0\end{align*}
for $\nu\in(\nu^{\star}+1,0).$ Here we used the Rayleigh sums \cite[p. 502]{watson}
$$\sum_{n\geq 1}\frac{1}{j_{\nu,n}^4}=\frac{1}{16(\nu+1)^2(\nu+2)}, \ \ \sum_{n\geq 1}\frac{1}{j_{\nu,n}^6}=\frac{1}{32(\nu+1)^3(\nu+2)(\nu+3)},$$
and the fact that function $f:(-1,0)\to\mathbb{R},$ defined by
$$f(x)=-\log 2-\log(x+1)+\frac{1}{2(x+1)}+\frac{3x+5}{4(x+1)^2(x+2)(x+3)},$$
is decreasing as the sum of three decreasing functions, and thus $f(x)>f(0)\simeq0.0151{\dots}>0$ if $x\in(-1,0).$

We would like to note that there is another way to prove that $J_{\nu+1}(1)>0$ if $\nu>\nu^{\star}.$ Since we have $2^{\nu+1}\Gamma(\nu+2)>0$ for $\nu\in(-2,\infty)$ it follows that $J_{\nu+1}(1)$ and $2^{\nu+1}\Gamma(\nu+2)J_{\nu+1}(1) $ have the same sign and the same roots in the interval  $(-2,\infty).$ On the other hand, we have  $$h_{\nu+1}(1)=2^{\nu+1}\Gamma(\nu+2)J_{\nu+1}(1)=1+\sum_{n\geq1}\frac{(-1)^n}{4^nn!(\nu+2)\dots(\nu+n+1)}.$$
Let $\Lambda:(-2,\infty)\rightarrow\mathbb{R}$  be the function defined by
$$\Lambda(\nu)=h_{\nu+1}(1)=1+\sum_{n\geq1}\frac{(-1)^n}{4^nn!(\nu+2)\dots(\nu+n+1)}.$$
We will show that $\Lambda$  is strictly  increasing. We have
\begin{align*}
\frac{d\Lambda(\nu)}{d\nu}&=\sum\limits_{n\geq1}\frac{(-1)^{n-1}}{4^nn!}\frac{\sum\limits_{k=1}^n\frac{1}{\nu+k+1}}{(\nu+2)\dots(\nu+n+1)}\\
&=\sum_{n\geq0}\left[\frac{\sum\limits_{k=1}^{2n+1}\frac{1}{\nu+k+1}}{4^{2n+1}(2n+1)!(\nu+2)\dots(\nu+2n+2)}-\frac{\sum\limits_{k=1}^{2n+2}\frac{1}{\nu+k+1}}{4^{2n+2}(2n+2)!(\nu+2)\dots(\nu+2n+3)}\right]\\
&=\sum_{n\geq0}\frac{1}{4^{2n+2}(2n+2)!\prod\limits_{k=1}^{2n+1}(\nu+k+1)}\left[\sum\limits_{k=1}^{2n+1}\frac{1}{\nu+k+1}-\frac{1}{4(2n+2)(2n+3)}\sum_{k=1}^{2n+2}\frac{1}{\nu+k+1}\right].
\end{align*}
The last expression is certainly  positive because the next inequalities hold for $n\in\{1,2,\dots\}$ and $\nu>-2$
$$\sum_{k=2}^{2n+1}\frac{1}{\nu+k+1}>\frac{1}{4(2n+2)(2n+3)}\sum\limits_{k=2}^{2n+1}\frac{1}{\nu+k+1},$$
and
$$\frac{1}{\nu+2}-\frac{1}{4(2n+2)(2n+3)}\frac{1}{\nu+2}>\frac{1}{2}\frac{1}{\nu+2}>\frac{1}{4(2n+2)(2n+3)}\frac{1}{\nu+2n+3}.$$
Consequently we have that $\Lambda$  is a strictly increasing function, and this implies that the equation $J_{\nu+1}(1)=0$ has a unique root in the interval   $(-2,\infty).$ This root is $\nu^\star=-1.7744{\dots}.$ It follows also that if $\nu>\nu^\star$ then $\Lambda(\nu)>\Lambda(\nu^\star)=0,$ or equivalently, $J_{\nu+1}(1)>0$ if $\nu>\nu^{\star}.$

By using \eqref{varco} for $r=1$ and the fact that the function $\psi$ is strictly decreasing, we get that
$$\real\left(1+z\frac{\varphi_{\nu}''(z)}{\varphi_{\nu}'(z)}\right)\geq 1-\sum_{n\geq1}\frac{1}{j_{\nu+1,n}^2-1}=1+\frac{\varphi_{\nu}''(1)}{\varphi_{\nu}'(1)}$$
for all $z\in\mathbb{D}.$ Since according to \eqref{gyokder} for each $n\in\{1,2,\dots\}$ the function $\nu\mapsto j_{\nu,n}$ is strictly increasing on $(-1,\infty),$ it follows that the function $\nu\mapsto j_{\nu+1,n}$ is strictly increasing on $(-2,\infty)$ for each $n\in\{1,2,\dots\}.$ Consequently, the function $\phi:(\nu^{\star},\infty)\rightarrow\mathbb{R},$ defined by
$$\phi(\nu)=\frac{J_{\nu}(1)}{2J_{\nu+1}(1)}-\nu=1-\sum_{n\geq1}\frac{1}{j_{\nu+1,n}^2-1},$$
is strictly increasing too. Note that the function $\phi$ is well defined since for $\nu>\nu^{\star}$ we have $J_{\nu+1}(1)\neq0.$ Now, using the fact that $\phi$ is strictly increasing we obtain that if $\nu\geq\nu_{\alpha}=\nu_{\alpha}(\varphi_{\nu}),$ then the next inequality is valid
\begin{equation}\label{kutya}
\frac{J_{\nu}(1)}{2J_{\nu+1}(1)}-\nu=\phi(\nu)\geq{\phi}(\nu_{\alpha})=\frac{J_{\nu_{\alpha}}(1)}{2J_{\nu_{\alpha}+1}(1)}-\nu=\alpha.
\end{equation}
Now, from (\ref{kutya}) we get that $\nu_{\alpha}$ is the smallest value having the property that the condition $\nu\geq\nu_{\alpha}$ implies
$$\real\left(1+z\frac{\varphi_{\nu}''(z)}{\varphi_{\nu}'(z)}\right)>\frac{J_{\nu_{\alpha}}(1)}{2J_{\nu_{\alpha}+1}(1)}-\nu=\alpha \ \ \ \mbox{for all}\ \ \ z\in\mathbb{D}.$$

In other words, we proved that the function $\varphi_{\nu}$ is convex of order $\alpha\in[0,1)$ in $\mathbb{D}$ if and only if
$\nu\geq\nu_{\alpha}(\varphi_{\nu}),$ where $\nu_{\alpha}=\nu_{\alpha}(\varphi_{\nu})$ is the unique  root of
the equation
$$
 \frac{J_{\nu}(1)}{2J_{\nu+1}(1)}-\nu=\alpha
$$
situated in $\left(-{2},\infty\right).$ In particular, $\varphi_{\nu}$ is convex if and only if
$\nu\geq\nu_0=\nu_0(\varphi_{\nu}),$ where $\nu_0\simeq-1.5623\dots$ is the unique root of the equation
$J_{\nu}(1)=2\nu J_{\nu+1}(1).$
\end{proof}

\subsection*{Added in proof} Recently, Baricz et al. \cite{dini} presented an alternative proof of Lemma \ref{lem4} by using the Hadamard theorem concerning the growth order of entire functions. Moreover, Baricz and Sz\'asz \cite{close} presented an alternative proof of the convexity of $h_{\nu}$ by using a result of Shah and Trimble concerning transcendental entire functions.

\end{document}